\documentclass{amsart}
\usepackage[all,2cell]{xy}
\usepackage{amssymb,amsmath,stmaryrd,mathrsfs,amsthm}
\usepackage{eucal}
\usepackage{hyperref}

\newcommand{\sA}{\ensuremath{\mathscr{A}}}
\newcommand{\sB}{\ensuremath{\mathscr{B}}}
\newcommand{\sC}{\ensuremath{\mathscr{C}}}
\newcommand{\sD}{\ensuremath{\mathscr{D}}}

\newcommand{\sI}{\ensuremath{\mathscr{I}}}
\newcommand{\sK}{\ensuremath{\mathscr{K}}}
\newcommand{\sO}{\ensuremath{\mathscr{O}}}
\newcommand{\sU}{\ensuremath{\mathscr{U}}}
\newcommand{\calA}{\ensuremath{\mathcal{A}}}
\newcommand{\calH}{\ensuremath{\mathcal{H}}}

\newcommand{\calS}{\ensuremath{\mathcal{S}}}
\newcommand{\bbN}{\ensuremath{\mathbb{N}}}
\newcommand{\Atil}{\ensuremath{\widetilde{A}}}

\newcommand{\fhat}{\ensuremath{\widehat{f}}}
\newcommand{\uG}{\underline{\Gamma}}
\newcommand{\bfL}{\ensuremath{\mathbf{L}}}
\newcommand{\bfR}{\ensuremath{\mathbf{R}}}
\newcommand{\bfM}{\ensuremath{\mathbf{M}}}

\newcommand{\HoSh}{\operatorname{HoSh}}
\newcommand{\ten}{\ensuremath{\otimes}}
\newcommand{\Id}{\ensuremath{\operatorname{Id}}}
\newcommand{\Ho}{\ensuremath{\operatorname{Ho}}}
\newcommand{\Map}{\ensuremath{\operatorname{Map}}}
\newcommand{\op}{\ensuremath{^{\mathit{op}}}}
\newcommand{\adj}{\dashv}
\newcommand{\iso}{\cong}
\newcommand{\eqv}{\simeq}
\newcommand{\too}[1][]{\ensuremath{\overset{#1}{\longrightarrow}}}
\newcommand{\toot}{\ensuremath{\rightleftarrows}}
\newcommand{\into}{\ensuremath{\hookrightarrow}}
\newcommand{\we}{\ensuremath{\overset{\sim}{\longrightarrow}}}
\newcommand{\maps}{\colon}
\newcommand{\spam}{\,:\!}
\newcommand{\Set}{\ensuremath{\mathit{Set}}}

\theoremstyle{plain}
\newtheorem{thm}{Theorem}[section]
  
\newtheorem{cor}{Corollary}
  \makeatletter\let\c@cor\c@thm\makeatother
  \numberwithin{cor}{section}
  
\newtheorem{prp}{Proposition}
  \makeatletter\let\c@prp\c@thm\makeatother
  \numberwithin{prp}{section}
  
\newtheorem{lemma}{Lemma}
  \makeatletter\let\c@lemma\c@thm\makeatother
  \numberwithin{lemma}{section}
  
\newtheorem{assume}{Assumption}
  \makeatletter\let\c@assume\c@thm\makeatother
  \numberwithin{assume}{section}
  
\theoremstyle{definition}
\newtheorem{defn}{Definition}
  \makeatletter\let\c@defn\c@thm\makeatother
  \numberwithin{defn}{section}
  
\theoremstyle{remark}
\newtheorem{rmk}{Remark}
  \makeatletter\let\c@rmk\c@thm\makeatother
  \numberwithin{rmk}{section}
  
\newtheorem{ceg}{Counterexample}
  \makeatletter\let\c@ceg\c@thm\makeatother
  \numberwithin{ceg}{section}

\renewcommand{\theenumi}{(\roman{enumi})}

\makeatletter
\let\c@equation\c@thm
\makeatother
\numberwithin{equation}{section}

\numberwithin{thm}{section}

  \title{Parametrized spaces model locally constant homotopy sheaves}
  \author{Michael~A.\ Shulman}
  \address{Department of Mathematics, University of Chicago\\
    5734 S. University Ave, Chicago IL 60615}
\begin{document}
\maketitle

\begin{abstract}
  We prove that the homotopy theory of parametrized spaces embeds
  fully and faithfully in the homotopy theory of simplicial
  presheaves, and that its essential image consists of the locally
  homotopically constant objects.  This gives a homotopy-theoretic
  version of the classical identification of covering spaces with
  locally constant sheaves.  We also prove a new version of the
  classical result that spaces parametrized over $X$ are equivalent to
  spaces with an action of $\Omega X$.  This gives a
  homotopy-theoretic version of the correspondence between covering
  spaces and $\pi_1$-sets.  We then use these two equivalences to
  study base change functors for parametrized spaces.
\end{abstract}

\tableofcontents

\section{Introduction}
\label{sec:introduction}

Recently there has been growing interest in doing homotopy theory
`relative' to a base topological space.  One motivation for this is to
find a framework which includes both local cohomology and generalized
cohomology theories, since clearly such a generalization requires a
notion of `spectrum relative to a base space', or at least of `space
relative to a base space'.  In this paper we focus on spaces for
simplicity; we hope to deal with spectra in a later paper.

There are two general approaches to such a relative theory in the
literature: one involving `sheaves of spaces on $B$', or
\emph{homotopy sheaves} (also known as \emph{stacks}), such as that
of~\cite{jardine:simplicial-presheaves,lurie:higher-topoi}, and one
involving `spaces over $B$', or \emph{parametrized spaces}, such as
that of~\cite{maysig:pht}.  Formal comparisons of the two, however,
are difficult to find in the literature.  In this paper, we state and
prove such a comparison; our slogan is that \emph{parametrized spaces
  are equivalent to locally constant homotopy sheaves}.

Our inspiration comes from the well-known equivalence between the
following three categories.
\begin{enumerate}\renewcommand{\theenumi}{(\roman{enumi})}
\item Locally constant sheaves of sets on $B$.\label{item:lc-sh}
\item Covering spaces over $B$ (which are fibrations with discrete
  fibers).\label{item:cov-sp}
\item Sets with an action of $\pi_1(B)$.  If $B$ is not
  path-connected, we use instead the fundamental groupoid
  $\Pi_1(B)$.\label{item:pi1-sets}
\end{enumerate}
Our goal is to prove a `homotopical' version of this.  Specifically,
we will show that the following three homotopy theories are
equivalent.
\begin{enumerate}\renewcommand{\theenumi}{(\alph{enumi})}
\item Homotopy sheaves on $B$ which are `locally constant'.\label{item:lc-hosh}
\item Fibrations over $B$.\label{item:param-sp}
\item (If $B$ is path-connected) spaces with an action of $\Omega B$.
  We regard $\Omega B$ as representing the automorphisms of the base
  point of `$\Pi_\infty(B)$', the `fundamental $\infty$-groupoid' of
  $B$.\label{item:piinf-sp}
\end{enumerate}
Often, of course, we use a larger category of models.  We find the
homotopy sheaves as the fibrant objects in a model structure on the
category of simplicial presheaves, and the fibrations over $B$ as the
fibrant objects in a model structure on the category of all spaces
over $B$.  We also refer to this latter as the homotopy theory of
\emph{parametrized spaces}.

Our method of proof is also similar to the `0-dimensional' version.
One way to prove the equivalence between~\ref{item:lc-sh}
and~\ref{item:cov-sp} is to first prove that the category of all
sheaves of sets on $B$ is equivalent to the category of local
homeomorphisms (or `etale spaces') over $B$, and then identify the
covering spaces as the local homeomorphisms which are `locally
constant'.  Analogously, we will prove the equivalence
between~\ref{item:lc-hosh} and~\ref{item:param-sp} by using a
different model structure on spaces over $B$, due
to~\cite{ij:ex-spaces}, whose homotopy theory is equivalent to that of
homotopy sheaves and in which all objects are fibrant.  We will show
that a model structure for spaces parametrized over $B$ embeds into
this model structure, and that its image consists of the `locally
constant' homotopy sheaves.


Likewise, the equivalence between~\ref{item:cov-sp}
and~\ref{item:pi1-sets} goes by taking the fiber of a covering space,
with action induced by path-lifting around loops in $B$.  We prove the
equivalence between~\ref{item:param-sp} and~\ref{item:piinf-sp} using
a homotopical version of this.

Our motivating analogy also suggests other aspects of the relationship
between homotopy sheaves and parametrized spaces.  For example, since
covering spaces over $B$ are equivalent to $\pi_1(B)$-sets, they
depend only on homotopy-theoretic information about $B$, while the
category of all sheaves of sets on $B$ determines $B$ essentially up
to homeomorphism.  Analogously, the homotopy theory of parametrized
spaces is invariant under weak equivalences of the base space, while
that of homotopy sheaves is not.  This is not a problem with either
approach, merely a difference in emphasis: homotopy theorists are only
interested in spaces as homotopy types, while sheaf theorists are
interested in spaces, such as spectra of rings, which carry more
information than their ordinary weak homotopy type can support.

Another important difference has to do with base change functors and
homology and cohomology.  Under the correspondence
between~\ref{item:lc-sh} and~\ref{item:cov-sp}, the sheaf cohomology
of a locally constant sheaf of groups on $B$ is identified with the
local cohomology of $B$ with coefficients in the corresponding local
system.  However, while it is easy to also define \emph{homology} with
local coefficients, it is quite difficult to define `sheaf homology'
in general, and this difference carries over to the homotopical
version.

The analogues of homology and cohomology in relative homotopy theory
are, respectively, derived left and right adjoints $f_!$ and $f_*$ to
the pullback functor $f^*$ for a map $f$ of base spaces; when $f$ is
the projection $r\maps B\to *$, we expect to recover homology from
$r_!$ and cohomology from $r_*$.  For homotopy sheaves, the adjunction
$f^*\adj f_*$ is well-behaved, but in general $f^*$ has no left
adjoint.  For parametrized spaces, on the other hand, the adjunction
$f_!\adj f^*$ is well-behaved, while the right adjoint $f_*$ is harder
to get a handle on.  A right adjoint $f_*$ on the homotopy-category
level was shown to exist in~\cite{maysig:pht} only by using Brown
representability, and only on connected spaces.

One motivation for our comparison result is the hope to shed some
light on the right adjoint $f_*$ for parametrized spaces.  We will
show that the derived functor $f^*$ for parametrized spaces agrees
with the derived functor $f^*$ for the corresponding locally constant
homotopy sheaves; in particular, the $f^*$ for homotopy sheaves
preserves locally constant objects.  The functors $f_*$, on the other
hand, agree whenever $f$ is a fibration between locally compact CW
complexes, but in general the $f_*$ for homotopy sheaves need not
preserve locally constant objects.

It follows that for such fibrations, the $f_*$ for parametrized spaces
can be computed by passing through homotopy sheaves.  This is not a
huge gain in generality, since $f_*$ can be computed by existing
methods when $f$ is a bundle of cell complexes, but we give some
motivation for believing that it is almost best possible.  We also
give examples in which the $f_*$ for homotopy sheaves is very
different from the $f_*$ for parametrized spaces.

The equivalence between~\ref{item:param-sp} and~\ref{item:piinf-sp} is
more promising for the construction of $f_*$, at least at a formal
level.  We will show that this equivalence preserves all the base
change functors, and in particular that a derived right adjoint $f_*$
can always be constructed for parametrized spaces by passing through
$\Omega B$-spaces.  This requires no assumptions on the map $f$ and no
connectivity assumptions on the spaces involved.  However, this
equivalence involves a chain of two adjunctions in different
directions, so to actually compute $f_*$ in this way may be
impractical.

The plan of this paper is as follows.  In
\S\S\ref{sec:point-set-topology}--\ref{sec:model-struct-param} we
recall some facts about point-set topology and model structures for
parametrized spaces.
\S\S\ref{sec:model-struct-homot}--\ref{sec:base-change} are then
devoted to the equivalence between~\ref{item:lc-hosh}
and~\ref{item:param-sp}.  In \S\ref{sec:model-struct-homot} we define
a model structure for homotopy sheaves using simplicial presheaves and
compare it to the model structure from~\cite{ij:ex-spaces} which uses
actual spaces over a base space.  In \S\ref{sec:comparison}, we prove
that the homotopy theory of parametrized spaces embeds into that of
homotopy sheaves, and in \S\ref{sec:essimg} we prove that the image
consists of the `locally constant' homotopy sheaves.  Then in
\S\ref{sec:base-change} we compare the base change functors in the two
situations.
  
\S\S\ref{sec:g-spaces-bg}--\ref{sec:base-change-g} deal with the
equivalence between~\ref{item:param-sp} and~\ref{item:piinf-sp}.  In
\S\ref{sec:g-spaces-bg} we prove that when $G$ is a grouplike
topological monoid, such as a Moore loop space $\Omega A$, the
homotopy theories of $G$-spaces (with the underlying weak
equivalences, not the weak equivalences usually used in equivariant
homotopy theory) and of spaces parametrized over $BG$ are equivalent.
Since any connected space $A$ is weakly equivalent to $B(\Omega A)$,
and parametrized spaces are invariant under weak equivalence of the
base space, this shows that spaces over $A$ are equivalent to $\Omega
A$-spaces.  Finally, in \S\ref{sec:base-change-g} we show that this
equivalence preserves all the base change functors.

An important technical tool in our work is a new model structure for
topological spaces discovered by Cole~\cite{cole:mixed}, obtained by
\emph{mixing} the `standard' model structure constructed by
Quillen~\cite{quillen:htpical-alg} with the `classical' model
structure constructed by Str\o m~\cite{strom:the-htpy-cat}.  In
Cole's model structure the weak equivalences are the \emph{weak}
homotopy equivalences, while the fibrations are the \emph{Hurewicz}
fibrations, and the cofibrant objects are the spaces of the homotopy
type of a CW complex.  This model structure is arguably closer to
classical homotopy theory than is the standard model structure, and
its cofibrant and fibrant objects are also better behaved and
preserved by more constructions.

Results analogous to ours can be found in~\cite{toen:locconst}, which
works with simplicial sets and Kan fibrations rather than topological
spaces and Hurewicz fibrations.  This provides another formalization
of the slogan that parametrized spaces model locally constant homotopy
sheaves.  However, a topological approach is of independent interest
for many reasons.

I would like to thank my advisor, Peter May, for useful conversations
and for suggesting an improvement of \autoref{g-bg-equiv}; Mark
Johnson, for several helpful comments; and the anonymous referee, for
helpful suggestions on exposition.

\section{Point-set topology}
\label{sec:point-set-topology}

In the parametrized world there are always some point-set topological
issues that must be dealt with.  It is by now generally accepted that
a good category of topological spaces for homotopy theory must be
cartesian closed, and the most common choice is the category of
\emph{compactly generated} spaces; that is, weak Hausdorff $k$-spaces
(see~\cite[Ch.~5]{may:concise}).  However, in the parametrized case
one wants the category of spaces \emph{over} every base space $B$ to
also be cartesian closed.  Standard categorical arguments show that
this is equivalent to the existence, for any map $f\maps A\to B$ of
base spaces, of a right adjoint $f_*$ to the pullback functor $f^*$.

However, this extra desideratum is false for compactly generated
spaces.  Various remedies are possible.  One is to restrict the
structure maps $X\to B$ of spaces over $B$, and the transition maps
$f\maps A\to B$ of base spaces, to be open maps, as is done
in~\cite{ij:ex-spaces}.  However, in some cases this is too
restrictive; for example, it disallows diagonal maps $\Delta\maps B\to
B\times B$.  Another solution is to use a topological quasitopos, such
as pseudotopological spaces (see~\cite{wyler:quasitopoi}) or
subsequential spaces (see~\cite{ptj:topological-topos}).

We adopt instead the solution used in~\cite{maysig:pht}: we require
base spaces to be compactly generated, but allow total spaces to be
arbitrary $k$-spaces, not necessarily weak Hausdorff.  The references
given in~\cite[\S1.3]{maysig:pht} show that if $B$ is compactly
generated, the category $\sK/B$ of $k$-spaces over $B$ is cartesian
closed, and if $f\maps A\to B$ is a continuous map between compactly
generated spaces, the pullback functor $f^*\maps \sK/B\to\sK/A$ has
not only a left adjoint $f_!$ but a right adjoint $f_*$.  The same is
true if we consider the categories $\sK_B$ of sectioned spaces over
$B$.

The left adjoint $f_!\maps \sK/A\to\sK/B$ is simply given by
composition with $f$.  In the sectioned case, $f_!\maps \sK_A\to\sK_B$
is defined by a pushout, which in the case of the map $r\maps B\to *$
simply quotients out the section.

We think of the right adjoint $f_*$ as a `space of relative sections'.
When $f$ is the map $A\to *$, the space $f_* X$ is simply the space of
global sections of $X\too[p] A$; that is, the subspace of $\sK(A,X)$
consisting of the maps $A\too[s] X$ such that $ps=1_A$.

From now on, when we speak of \emph{a space over $B$} it is to be
understood that $B$ is compactly generated and the total space is a
$k$-space.  Although our point-set conventions are different than
those of~\cite{ij:ex-spaces}, it is readily seen that all the proofs
in~\cite{ij:ex-spaces} carry over without difficulty to our setting,
so this will be our last comment on the difference.

One fundamental result we will need is the following.  Recall (for
example, from~\cite[\S1.5]{fp:cellular-structures}) that the following
conditions on a CW complex $X$ are equivalent.
\begin{enumerate}
\item $X$ is \textbf{locally finite}, meaning that each point has a
  neighborhood which intersects only finitely many cells.
\item $X$ is locally compact.
\item $X$ is metrizable.
\item $X$ is first countable.
\end{enumerate}

\begin{thm}[{\cite[6.6]{ij:ex-spaces}}]\label{locfincw->openhtcw}
  If $X$ is a CW complex with the above properties, then any open
  subspace of $X$ has the homotopy type of a CW complex.
\end{thm}
\begin{proof}
  Given a class $\sC$ of spaces, in~\cite{hyman:m-spaces} a space is
  defined to be an \emph{ANR(\sC)} (absolute neighborhood retract) if
  it is a neighborhood retract of every space in \sC\ that contains it
  as a closed subset.  If \sC\ is the class of metric spaces, an
  ANR(\sC) is called a \emph{metric ANR} or just an \emph{ANR}.
  By~\cite[11.4]{hyman:m-spaces}, every CW complex is an ANR(M), where
  M is the class of `M-spaces' defined there.  The remarks
  before~\cite[10.4]{hyman:m-spaces} show that an ANR(M) is a metric
  ANR just when it is metrizable.

  Thus, since our CW complex $X$ is assumed metrizable, it is an ANR.
  But by~\cite[A.6.4]{fp:cellular-structures}, any open subset of an
  ANR is an ANR, and by~\cite[5.2.1]{fp:cellular-structures} spaces of
  the homotopy type of CW complexes coincide with spaces of the
  homotopy type of ANRs.
\end{proof}

This is important because in order to compare spaces over $X$ to
sheaves on $X$, we need to consider \emph{sections} over open subsets
$U\subset X$.  Of course, sections over $U$ are particular maps out of
$U$, and we know that only spaces of the homotopy type of CW complexes
are `homotopically good' for mapping out of.

\section{Model structures for parametrized spaces}
\label{sec:model-struct-param}

There are several model structures on the category $\sK/B$ of spaces
over $B$.  Any model structure on \sK\ gives rise, by standard
arguments, to a model structure on $\sK/B$.  The most well-known model
structures on \sK\ are the following.
\begin{enumerate}
\item The \textbf{standard} or $q$-model structure, in which the weak
  equivalences are the weak homotopy equivalences, the fibrations are
  the Serre fibrations, and the cofibrations are the retracts of
  relative cell complexes.  This is the model structure originally
  constructed by Quillen in~\cite{quillen:htpical-alg}.
\item The \textbf{classical} or $h$-model structure, in which the weak
  equivalences are the homotopy equivalences, the fibrations are the
  Hurewicz fibrations (or `$h$-fibrations'), and the cofibrations are
  the closed Hurewicz cofibrations.  This model structure was
  constructed in~\cite{strom:the-htpy-cat}.
\end{enumerate}
However, as mentioned in \S\ref{sec:introduction}, there is also a
\textbf{mixed} model structure, which was discovered by Cole.

\begin{thm}[\cite{cole:mixed}]\label{thm:mixed}
  Suppose that a category \sC\ has two model structures, called the
  $q$-model structure and the $h$-model structure, such that
  \begin{itemize}
  \item Every $h$-equivalence is a $q$-equivalence, and
  \item Every $h$-fibration is a $q$-fibration.
  \end{itemize}
  Then \sC\ also has a \textbf{mixed} or \textbf{$m$-model structure}
  in which
  \begin{itemize}
  \item The weak equivalences are the $q$-equivalences,
  \item The fibrations are the $h$-fibrations,
  \item The cofibrations are the $h$-cofibrations which factor as a
    $q$-cofibration followed by an $h$-equivalence, and the
    $m$-cofibrant objects are the $h$-cofibrant objects which have the
    $h$-homotopy type of a $q$-cofibrant object.
  \end{itemize}
  If the $q$-model structure is left or right proper, so is the
  $m$-model structure.  If the $q$- and $h$-model structures are both
  monoidal, so is the $m$-model structure.
\end{thm}

As is evident, we prefix model-theoretic words like `equivalence',
`fibration', `cofibration', and `cofibrant' with a letter to indicate
which model structure we are referring to.  However, we continue to
refer to $q$-equivalences and $h$-fibrations rather than
$m$-equivalences and $m$-fibrations.  If we want to make clear which
model structures are being mixed, we may refer to the mixing of the
$q$- and $h$-model structures as the $m_{q,h}$-model structure.

In the case of \sK, the standard and classical model structures mix to
give a model structure in which the weak equivalences are the
\emph{weak} homotopy equivalences and the fibrations are the Hurewicz
fibrations.  A map $f\maps A\to X$ is an $m$-cofibration if and only
if it is a Hurewicz cofibration that is cofiber homotopy equivalent
under $A$ to a relative CW complex.  In particular, the $m$-cofibrant
objects are the spaces of the homotopy type of a CW complex; thus we
can rephrase \autoref{locfincw->openhtcw} by saying that a locally
compact CW complex is \textbf{hereditarily $m$-cofibrant}.  Since the
$q$- and $h$-model structures on \sK\ are monoidal, so is the
$m$-model structure.

The mixed model structure packages classical information in an
abstract way.  For example, it is true in the generality of
\autoref{thm:mixed} that a $q$-equivalence between $m$-cofibrant
objects is an $h$-equivalence.  Note that the identity functor is a
Quillen equivalence between the $m$- and $q$-model structures, and
that unlike the $q$-model structure, neither the $h$- nor the
$m$-model structure on \sK\ is cofibrantly generated.

We denote the model structures on $\sK/B$ obtained from the $h$, $q$,
and $m$-model structures on \sK\ by the same letters.  Note that the
$m$-model structure on $\sK/B$ induced by the $m$-model structure on
\sK\ is the same as the model structure obtained by mixing the $q$-
and $h$-model structures on $\sK/B$.

We also have a \textbf{fiberwise} or $f$-model structure on $\sK/B$,
whose weak equivalences are the \emph{fiberwise} homotopy equivalences
(hereafter \emph{$f$-equivalences}); see~\cite[\S5.1]{maysig:pht} for
more details.  However, the $q$-model structure does not mix with the
$f$-model structure, since not every $f$-fibration is a $q$-fibration.

The homotopy theory on $\sK/B$ we are interested in is that modeled by
the Quillen equivalent $q$- and $m$-model structures.  It was observed
in~\cite{maysig:pht} that the $q$-model structure on $\sK/B$ is not
good enough for some purposes because it has too many cofibrations,
and of course the $m$-structure has even more.  Thus, a main technical
result of~\cite{maysig:pht} was the construction of a Quillen
equivalent `$qf$-model structure' with better formal properties.  The
$qf$-structure will not play any role for us, however, since we will
be more interested in controlling the fibrations than the
cofibrations.  For this, the best choice is the $m$-model structure,
in which the fibrant objects are the Hurewicz fibrations over $B$.

By standard arguments, each model structure on $\sK/B$ gives rise to a
corresponding model structure on the category $\sK_B$ of
\emph{sectioned} spaces over $B$, since the latter is just the
category of pointed objects (that is, objects under the terminal
object) in $\sK/B$.  In $\sK_B$ one may also consider `fiberwise
pointed' homotopy equivalences, fibrations, and so on; these form an
`$fp$-model structure' in the category $\sU_B$ of compactly generated
spaces over $B$, but it is unknown whether they do so in $\sK_B$;
see~\cite[5.2.9]{maysig:pht}.

Recall that for any map $f\maps A\to B$, we have a string of
adjunctions $f_!\adj f^*\adj f_*$ at the point-set level.

\begin{prp}[{\cite[\S7.3]{maysig:pht}}]
  The adjunction $f_!\adj f^*$ is Quillen for the $q$- and $m$-model
  structures.  If $f$ is a $q$-equivalence, then $f_!\adj f^*$ is a
  Quillen equivalence.  If $f$ is a bundle whose fibers are cell
  complexes, then $f^*\adj f_*$ is Quillen for the $q$-model
  structures.
\end{prp}

The results in~\cite{maysig:pht} are only stated for the sectioned
case of $\sK_B$, but the proofs remain valid in the unsectioned case
of $\sK/B$.

This implies that we always have derived adjunctions $\bfL_q f_! \adj
\bfR_q f^*$ at the level of homotopy categories, and that when $f$ is
a bundle of cell complexes, we also have a derived adjunction $\bfL_q
f^*\adj \bfR_q f_*$.  Because in the latter case $f^*$ is left and
right Quillen for the same model structure, its left and right derived
functors agree.  We decorate \bfL\ and \bfR\ with a $q$ to remind us
that these are derived functors with respect to the $q$-equivalences;
since left and right derived functors are determined by the weak
equivalences of a model structure, the derived functors are the same
whether we use the $q$- or the $m$-model structures.

For maps $f$ other than bundles of cell complexes, the functor $f_*$
is difficult to get a handle on homotopically.  It is proven
in~\cite[9.3.2]{maysig:pht}, using Brown representability, that in the
sectioned case, for any map $f$ the functor $\bfR_q f^*$ has a partial
right adjoint defined on connected objects.  However, in general no
relationship between this functor and the point-set level functor
$f_*$ is known.  In \S\ref{sec:base-change} and
\S\ref{sec:base-change-g} we will see that this problem can be
partially remedied by passing across one or the other of our
equivalences.

\section{Model structures for homotopy sheaves}
\label{sec:model-struct-homot}

There are several ways to make the notion of `homotopy sheaves'
precise.  Probably the most common approach is the following.  Let $B$
be a space and let $\sB$ denote the poset of open sets in $B$.  We
write $\calS$ for the category of simplicial sets, equipped with its
usual model structure.  The category $\calS^{\sB\op}$ of simplicial
presheaves on \sB\ then has a \emph{projective} model structure in
which the weak equivalences and fibrations are objectwise.  We now
localize this structure at a suitable set of maps to obtain a new
model structure whose fibrant objects may be called `homotopy
sheaves'.

First we introduce some notation.  We write $y\maps \sB\into
\Set^{\sB\op}$ for the Yoneda embedding, so that $yU$ is the presheaf
represented by an open subset $U\subset B$, i.e.\
\[yU(V) =
\begin{cases}
  \{\ast\} & \text{if } V\subset U\\
  \emptyset & \text{otherwise}.
\end{cases}
\]
If $U = \bigcup_{\alpha\in \calA} U_\alpha$ is an open cover of an
open subset $U\subset B$, we write $y\calA$ for the following
subfunctor of $yU$:
\[y\calA(V) =
\begin{cases}
  \{\ast\} & \text{if } V\subset U_\alpha \text{ for some }\alpha\in \calA \\
  \emptyset & \text{otherwise}.
\end{cases}
\]
We write $\sI_B$ for the set of all inclusions
\begin{equation}
  y\calA \into yU\label{eq:sieve}
\end{equation}
ranging over all open covers $U = \bigcup_{\alpha\in \calA} U_\alpha$
of open subsets $U\subset B$.

A presheaf of sets is a sheaf, in the usual sense, just when it sees
all the maps in $\sI_B$ as isomorphisms.  Thus, it makes sense to
localize $\calS^{\sB\op}$ at $\sI_B$ (considered as a set of maps
between discrete simplicial presheaves) and call the resulting model
structure the \textbf{homotopy sheaf model structure}.  A simplicial
presheaf is fibrant in this model structure when it is objectwise
fibrant and moreover sees all the maps $\sI_B$ as weak equivalences;
we call such an object a \textbf{homotopy sheaf}.  We denote the
homotopy category of this model structure by $\HoSh(B)$.

\begin{rmk}\label{top-hosh}
  We could also, if we wished, use the category $\sK^{\sB\op}$ of
  presheaves of topological spaces.  The standard Quillen equivalence
  \[|-|\maps \calS\toot\sK \spam S
  \]
  lifts to a Quillen equivalence between the projective model
  structures on $\calS^{\sB\op}$ and $\sK^{\sB\op}$, and thence to a
  Quillen equivalence between homotopy sheaf model structures.  In
  this paper we will use the simplicial version, because it is easier
  to write down explicit projective-cofibrant replacements.  However,
  in an equivariant context the discrete category \sB\ may need to be
  replaced by a topologically enriched category, in which case the use
  of spaces rather than simplicial sets would become important.
\end{rmk}

\begin{rmk}
  The above construction is the same idea followed
  in~\cite{lurie:higher-topoi}, although there the localization is
  done using quasi-categories rather than model categories.  However,
  the elements of the model-categorical approach can be found
  in~\cite[\S7.1]{lurie:higher-topoi}.  This model structure for
  homotopy sheaves is not equivalent to that
  of~\cite{jardine:simplicial-presheaves}, which is constructed by
  localizing with respect to the larger class of \emph{hypercoverings}
  (see~\cite{dhi:hypercovers}).  Several arguments for using coverings
  rather than hypercoverings can be found
  in~\cite{lurie:higher-topoi}, in particular the result we quote
  below as \autoref{hosheaf<->ijstr}.
\end{rmk}

As we mentioned in \S\ref{sec:introduction}, however, there is also a
model structure on the category $\sK/B$ which is Quillen equivalent to
the above simplicial model for homotopy sheaves.  This model structure
was called the \emph{fine} model structure in~\cite{ij:ex-spaces}
where it was first defined; an essentially identical model structure
was also constructed in~\cite[\S7.1.2]{lurie:higher-topoi}.  We will
call it the \textbf{$ij$-model structure}.

If $X\to B$ is a space over $B$ and $U\subset B$ is an open set, we
denote by $\uG(U,X)$ the \emph{space of sections of $X$ over $U$}.  It
can be defined as the mapping space $\Map_B(U,X)$ in the topologically
enriched category $\sK/B$, or more abstractly as $r_* j^* X$ where
$j\maps U\into X$ is the inclusion and $r\maps U\to *$ is the
projection.  The underlying sets $\Gamma(U,X)$ of the spaces
$\uG(U,X)$ form the ordinary sheaf of sections of $X$, but the spaces
of sections carry more information about the topology of $X$.  This
enables us, for instance, to distinguish between $X$ and the local
homeomorphism (or `etale space') corresponding to its ordinary sheaf
of sections.

We now define the following classes of maps.
\begin{itemize}
\item The $ij$-equivalences are the maps $f$ over $B$ such that
  $\uG(U,f)$ is a $q$-equivalence for all open sets $U\subset B$.
\item The $ij$-fibrations are the maps $f$ over $B$ such that
  $\uG(U,f)$ is a $q$-fibration for all open $U\subset B$.  In
  particular, every space over $B$ is $ij$-fibrant.
\item Of course, the $ij$-cofibrations are the maps over $B$ having
  the left lifting property with respect to the $ij$-trivial
  $ij$-fibrations.
\end{itemize}
It is proven in~\cite{ij:ex-spaces} that the above classes of maps
define a topological model structure on $\sK/B$.  It is clearly
cofibrantly generated; a set of generating cofibrations is
\[\big\{U\times S^{n-1} \into U\times D^n : n\in \bbN, U\subset B \text{ open}\big\}\]
and a set of generating trivial cofibrations is
\[\big\{U\times D^{n-1} \into U\times D^n : n\in \bbN, U\subset B \text{
  open}\big\}.
\]
Since the generating cofibrations are $f$-cofibrations and the
generating trivial cofibrations are $f$-trivial $f$-cofibrations, the
identity functor of $\sK/B$ is left Quillen from the $ij$-model
structure to the $f$-model structure.  Moreover, we have the following
fact.

\begin{lemma}
  Any $f$-equivalence is an $ij$-equivalence, and a map between
  $ij$-cofibrant objects is an $ij$-equivalence if and only if it is
  an $f$-equivalence.
\end{lemma}
\begin{proof}
  This follows from the fact that the model structure is topological
  (that is, it is a \sK-model category), using the topological version
  of~\cite[9.5.16]{hirschhorn:modelcats}.  Alternately, for the first
  statement one may observe that $r_*$ and $j^*$ both preserve
  homotopies.  Therefore, if $g$ is an $f$-equivalence, $\Gamma(U,g) =
  r_*j^*g$ is an $h$-equivalence and hence a $q$-equivalence for all
  open $U\subset X$, and thus $g$ is an $ij$-equivalence.  The second
  statement follows from this and the fact that the identity functor
  is left Quillen from the $ij$-model structure to the $f$-model
  structure.
\end{proof}

This implies that unlike the $q$-model structure, the $ij$-model
structure on $\sK/B$ \emph{does} mix with the $f$-model structure to
give a mixed $m_{ij,f}$-model structure.  We will make no essential use
of this model structure, but its existence is interesting.

Analogously, in the induced $ij$-model structure on $\sK_B$, any
$fp$-equivalence (in fact, any $f$-equivalence) is an
$ij$-equivalence, and a map between $ij$-cofibrant objects is an
$ij$-equivalence if and only if it is an $fp$-equivalence.  Recall,
though, that an `$fp$-model structure' is not known to exist on
$\sK_B$, so we do not have any `$m_{ij,fp}$-model structure'.


We now describe the equivalence between the $ij$-model structure and
the homotopy sheaf model structure.  There is a canonical adjoint pair
\begin{equation}
  |-|_B\maps \calS^{\sB\op} \toot \sK/B\spam S^B.\label{eq:rel-gr-sing}
\end{equation}
The right adjoint, called the \emph{relative singular complex}, is
defined by
\begin{equation}
  \label{eq:rel-sing}
  S^B(X)(U) = S\big(\uG(U,X)\big),
\end{equation}
where $S$ is the usual total singular complex of a space.  The left
adjoint $|-|_B$ is called the \emph{relative geometric realization};
it takes a simplicial presheaf $F$ to the tensor product of functors
$i \ten_{\sB} |F|$, where $|F|$ denotes the objectwise geometric
realization of $F$ and $i\maps \sB\to \sK/B$ sends each open set
$U\subset B$ to itself, considered as a space over $B$.

We say that a topological space is \emph{hereditarily paracompact} if
all its open subsets are paracompact.  This is true, for example, if
the space is metrizable.  Moreover, all CW complexes are hereditarily
paracompact (see~\cite[\S1.3]{fp:cellular-structures}).  The version
of the following result in~\cite{lurie:higher-topoi} applies more
generally, but we will only be interested in the hereditarily
paracompact case.

\begin{thm}[{\cite[7.1.4.5]{lurie:higher-topoi}}]\label{hosheaf<->ijstr}
  If $B$ is Hausdorff and hereditarily paracompact, then the
  adjunction~(\ref{eq:rel-gr-sing}) defines a Quillen equivalence
  between the homotopy sheaf model structure and the $ij$-model
  structure.  In particular, we have $\HoSh(B)\eqv\Ho_{ij}(\sK/B)$.
\end{thm}
\begin{proof}[Idea of proof]
  We will not give the whole proof, but we give enough of it to
  explain the need for the hypotheses on $B$.  By definition of
  $ij$-equivalences and $ij$-fibrations, the adjunction is Quillen for
  the projective model structure and the $ij$-model structure.  Thus,
  to show that it is Quillen for the homotopy sheaf model structure,
  it suffices to show that the left derived functor of $|-|_B$ (with
  respect to the projective model structure) takes the maps $\sI_B$ to
  $ij$-equivalences.

  To calculate $\bfL_{\mathit{proj}} |-|_B$, we must replace objects by
  projective-cofibrant ones.  The presheaf $yU$ is already
  projective-cofibrant, but $y\calA$ is not.  We can give an explicit
  description of a cofibrant replacement for $y\calA$ as follows:
  choose a total ordering of \calA, and define $\widetilde{y\calA}$ to
  be the geometric realization of the following simplicial object in
  $\calS^{\sB\op}$.
  \begin{equation}
    \xymatrix{
      \dots \ar@{.}[r] &
      \displaystyle\coprod_{\alpha \le \beta \le \gamma}^{\phantom{\alpha}}
      y(U_\alpha \cap U_\beta \cap U_\gamma)
      \ar[r] \ar@<-2mm>[r] \ar@<2mm>[r] &
      \displaystyle\coprod_{\alpha \le \beta}^{\phantom{\alpha}}
      y(U_\alpha \cap U_\beta)
      \ar@<-1mm>[l] \ar@<1mm>[l]
      \ar@<-2mm>[r] \ar@<2mm>[r] &
      \displaystyle\coprod_{\alpha}^{\phantom{\alpha}}
      yU_\alpha
      \ar[l]
    }\label{eq:proj-coft-sieve}
  \end{equation}
  Since $y\calA$ is the coequalizer of the last two face maps, it
  admits a map from $\widetilde{y\calA}$, which is a
  projective-cofibrant replacement.

  Now, the relative realization of $yU$ is just the space $U$ over
  $B$.  The relative realization of $\widetilde{y\calA}$ can be
  described as a subset of $B\times [0,1]^{\calA}$ by using
  barycentric coordinates in each simplex.  The points of
  $|\widetilde{y\calA}|_B$ are those pairs $(b,\phi)$, where $b\in B$
  and $\phi\maps \calA\to [0,1]$, such that
  \begin{itemize}
  \item $\phi(\alpha)>0$ for only finitely many $\alpha$,
  \item if $\phi(\alpha)>0$, then $b\in U_\alpha$, and
  \item $\sum_{\alpha} \phi(\alpha) = 1$.
  \end{itemize}
  The topology of $\widetilde{y\calA}$ is generally finer than that
  induced from $B\times [0,1]^{\calA}$, but this is largely irrelevant
  since the identity map is a homotopy equivalence between the two
  topologies; see~\cite[3.3.7]{fp:cellular-structures}.

  The map $|\widetilde{y\calA}|_B \to |yU|_B = U$ is the obvious
  projection.  A section of this projection over $B$ is precisely a
  partition of unity subordinate to the cover $(U_\alpha)$.  Since by
  assumption, $U$ is Hausdorff and paracompact, such a section exists,
  and a linear homotopy shows that it is actually the inclusion of a
  fiberwise deformation retract.  Since $f$-equivalences are
  $ij$-equivalences, we see that $\bfL_{\mathit{proj}}|-|_B$ takes the
  maps in $\sI_B$ to $ij$-equivalences, and hence the adjunction is
  Quillen for the homotopy sheaf model structure.

  Finally, the functor $S^B$ reflects weak equivalences by definition
  of the $ij$-equiv\-alences.  Thus, by~\cite[1.3.16]{hovey:modelcats},
  to obtain a Quillen equivalence it suffices to show that for any
  projective-cofibrant simplicial presheaf $X$, the map $X\to
  S^B|X|_B$ is an $\sI_B$-localization.  This is proven
  in~\cite[\S7.1.4]{lurie:higher-topoi} using another, more
  complicated, partition-of-unity argument.
\end{proof}

\begin{rmk}
  The preceeding proof breaks down if we localize $\calS^{\sB\op}$ at
  all hypercovers instead: the relative realization of a hypercover is
  not necessarily an $ij$-equivalence.
\end{rmk}

We now show that the base change functors in the two cases also agree.
Suppose that $f\maps A\to B$ is a continuous map, where $A$ and $B$
are Hausdorff and hereditarily paracompact.  Then, as observed
in~\cite[5.9]{ij:ex-spaces}, the adjunction
\begin{equation}
  f^*\maps \sK/B \toot \sK/A\spam f_*\label{eq:basechange-ij}
\end{equation}
is Quillen for the $ij$-structures, since $f^*$ preserves the
generating cofibrations and trivial cofibrations.  It thus gives rise
to a derived adjunction which we denote $\bfL_{ij} f^* \adj \bfR_{ij}
f_*$.

On the other hand, the functor $f^{-1}\maps \sB\to\sA$ induces, by
precomposition, a functor $f_*\maps \calS^{\sA\op} \to
\calS^{\sB\op}$, which has a left adjoint $f^*$ given by Kan
extension.

\begin{prp}
  The adjunction
  \begin{equation}
    f^*\maps \calS^{\sB\op}\toot \calS^{\sA\op}\spam f_*\label{eq:basechange-spshf}
  \end{equation}
  is Quillen for the homotopy sheaf model structures.
\end{prp}
\begin{proof}
  Since $f_*$ preserves objectwise fibrations and weak equivalences,
  the adjunction is Quillen for the projective model structures.
  Thus, by definition of Bousfield localization, it suffices to show
  that $\bfL_{\mathit{proj}}f^*$ takes the maps in $\sI_B$ to
  $\sI_A$-local equivalences.  However, since $f^*$ takes the
  representable functor $yU$ to $y(f^{-1}(U))$, for any cover
  $U=\bigcup U_\alpha$ in $B$ it takes the
  diagram~(\ref{eq:proj-coft-sieve}) to the corresponding diagram for
  the cover $f^{-1}(U) = \bigcup f^{-1}(U_\alpha)$ in $A$.  Since it
  also preserves colimits, it takes the resulting cofibrant
  replacement for a map in $\sI_B$ to the corresponding replacement
  for the corresponding map in $\sI_A$, which is clearly an
  $\sI_A$-local equivalence.
\end{proof}

Thus, we also have a derived adjunction $\bfL_{sh} f^* \adj \bfR_{sh}
f_*$.

\begin{thm}\label{basechange-hosheaf<->ijstr}
  The derived adjunctions of $f^*\adj f_*$ for the $ij$-model
  structure and the homotopy sheaf model structure agree under the
  Quillen equivalence~\eqref{eq:rel-gr-sing}.  More precisely, we have
  isomorphisms
  \begin{align*}
    \bfR_{sh} f_* \circ \bfR S^A &\iso
    \bfR S^B \circ \bfR_{ij} f_*\\
    \intertext{and}
    \bfL |-|_B \circ \bfL_{sh} f^* &\iso
    \bfL_{ij} f^* \circ \bfL |-|_A.
  \end{align*}
\end{thm}
\begin{proof}
  Since deriving Quillen adjunctions is functorial, it suffices to
  check that the point-set level adjunctions agree.  But if
  $X\in\sK/A$ and $U\in\sB$, we have $\uG(U,f_*X) \iso \uG(f^{-1}U,
  X)$, from which we see that $f_* \circ S^A \iso S^B \circ f_*$ as
  desired.  The other isomorphism follows formally.
\end{proof}

\begin{rmk}
  Of course, the functor $f^*\maps \sK/B\to\sK/A$ also has a left
  adjoint $f_!$.  It is observed in~\cite[5.9]{ij:ex-spaces} that when
  $f$ is an embedding, the adjunction $f_!\adj f^*$ is also Quillen
  for the $ij$-structures.  On the other hand, in general the functor
  $f^*\maps \calS^{\sB\op} \to \calS^{\sA\op}$ will not have a left
  adjoint at all.
\end{rmk}

By standard model-category arguments, the $ij$-structure and the
homotopy sheaf structure give rise to model structures on the
corresponding pointed categories $\sK_B$ and $\calS^{\sB\op}_*$.  The
following fact implies that \autoref{hosheaf<->ijstr} descends to the
pointed case as well.

\begin{prp}[{\cite[1.3.5 and 1.3.17]{hovey:modelcats}}]\label{qeqv->pted}
  If $F\maps \sC\toot \sD\spam G$ is a Quillen adjunction, there is a
  corresponding Quillen adjunction $F_*\maps \sC_*\too\sD_*\spam G_*$
  between the corresponding pointed model categories.  If in addition
  $F\adj G$ is a Quillen equivalence and the terminal object of \sC\
  is cofibrant and preserved by $F$, then $F_*\adj G_*$ is also a
  Quillen equivalence.
\end{prp}

\begin{cor}
  The sectioned adjunction
  \begin{equation}
    |-|_B\maps \calS^{\sB\op}_* \toot \sK_B\spam S^B.\label{eq:rel-gr-sing-sec}
  \end{equation}
  defines a Quillen equivalence between the model category of pointed
  homotopy sheaves and the sectioned $ij$-model structure.  For a map
  $f\maps A\to B$, the sectioned adjunctions $f^*\adj f_*$ are again
  Quillen in both cases and their derived functors agree under the
  equivalence~(\ref{eq:rel-gr-sing-sec}).
\end{cor}

\begin{rmk}
  The adjunction~(\ref{eq:rel-gr-sing}) actually factors through the
  category $\sK^{\sB\op}$ of topological homotopy sheaves:
  \[\calS^{\sB\op} \toot \sK^{\sB\op} \toot \sK/B,\]
  where the first adjunction is the Quillen equivalence from
  \autoref{top-hosh}.  By the 2-out-of-3 property for Quillen
  equivalences, it follows that the adjunction
  \[\sK^{\sB\op} \toot \sK/B\]
  is a Quillen equivalence between the topological homotopy sheaf
  model structure and the $ij$-model structure.  Analogous remarks
  apply to the base change functors and the pointed variants.
\end{rmk}

\section{Parametrized spaces embed in homotopy sheaves}
\label{sec:comparison}

We now want to show that the homotopy theory of parametrized spaces
embeds in that of homotopy sheaves.  First we introduce some
terminology.

\begin{defn}
  We say that a Quillen adjunction $F\maps \sC\toot \sD\spam G$ is a
  \textbf{right Quillen embedding} from \sD\ to \sC\ if, for any
  fibrant $Y\in \sD$, the canonical map
  \begin{equation}
    FQGY \too FGY \too Y\label{eq:counit-pointset}
  \end{equation}
  is a weak equivalence, where $Q$ denotes cofibrant
  replacement in \sC.
\end{defn}

We regard a right Quillen embedding as exhibiting the homotopy theory
of \sD\ as a `sub-homotopy-theory' of the homotopy theory of \sC.  Of
course, there is a dual notion of left Quillen embedding.  For
example, the identity functor of \sK\ is a left Quillen embedding from
the $q$- or $m$-model structure to the $h$-model structure.

It is well-known that a Quillen adjunction is a Quillen equivalence
just when it induces an equivalence of homotopy categories.  There is
an analogue for Quillen embeddings.

\begin{prp}
  A Quillen adjunction $F\adj G$ is a right Quillen embedding if and
  only if the right derived functor $\bfR G$ is full and faithful on
  homotopy categories.
\end{prp}
\begin{proof}
  $\bfR G$ is full and faithful just when the counit
  \begin{equation}
    \bfL F \circ \bfR G \too \Id_{\Ho\sD}\label{eq:counit-htpy}
  \end{equation}
  is an isomorphism, but~\eqref{eq:counit-htpy} is represented on the
  point-set level by~\eqref{eq:counit-pointset}, so the former is an
  isomorphism just when the latter is a weak equivalence.
\end{proof}

\begin{rmk}
  By~\cite[1.3.16]{hovey:modelcats}, a right Quillen embedding is a
  Quillen equivalence if and only if the left adjoint $F$ reflects
  weak equivalences between cofibrant objects.  This is a homotopical
  version of the fact that a full and faithful right adjoint is an
  equivalence if and only if its left adjoint reflects isomorphisms.
\end{rmk}

We are now working towards showing that the identity adjunction of
$\sK/B$ is a right Quillen embedding from the $m$-model structure
(that is, the $m_{q,h}$-model structure) to the $ij$-model structure.
We begin with a sequence of lemmas.

\begin{lemma}\label{lem:mstr-ijstr-quillen}
  If $B$ is hereditarily $m$-cofibrant, then the identity functor of
  $\sK/B$ is a left Quillen functor from the $ij$-model structure to
  the $m$-model structure.
\end{lemma}
\begin{proof}
  We must show that the generating cofibrations and trivial
  cofibrations for the $ij$-structure are $m$-cofibrations and
  $m$-trivial $m$-cofibrations.  However, the generating trivial
  cofibrations for the $ij$-structure are $f$-equivalences, and
  therefore $q$-equivalences, so it suffices to show that the
  generating $ij$-cofibrations are $m$-cofibrations.  Since the
  generating cofibrations have the form $U\times S^{n-1} \into U\times
  D^n$ for some open set $U\subset B$, and $U$ is $m$-cofibrant by
  assumption, this follows from the fact that the $m$-model structure
  is monoidal.
\end{proof}

In fact, as pointed out by the referee, the identity functor is also
left Quillen from the mixed $m_{ij,f}$-model structure to the
$m$-model structure.  This is because the $m$-fibrations are the
$h$-fibrations, which are also $f$-fibrations, and by
\autoref{lem:mstr-ijstr-quillen} all $m$-trivial $m$-fibrations are
$ij$-equivalences.  Thus we have the following diagram of identity
functors which are all left Quillen functors.  Both horizontal arrows
on the left are Quillen equivalences, and both horizontal arrows on
the right are left Quillen embeddings.
\[\xymatrix{(\sK/B,ij) \ar[r] &
  (\sK/B, m_{ij,f}) \ar[r] \ar[d] &
  (\sK/B, f) \ar[d] \\
  (\sK/B, q)\ar[r] &
  (\sK/B, m_{q,h}) \ar[r] &
  (\sK/B, h)}\]

\begin{lemma}\label{lem:prod-ij-coft}
  Let $B$ be hereditarily $m$-cofibrant and let $QF\to F$ be a
  $q$-cofibrant replacement of $F\in\sK$.  Then
  \begin{equation}
    B\times QF \too B\times F\label{eq:product-ij-repl}
  \end{equation}
  is an $ij$-cofibrant replacement of the product projection $B\times
  F\to B$.
\end{lemma}
\begin{proof}
  It is clear that $B\times QF$ is $ij$-cofibrant, since any
  decomposition of $QF$ into cells $S^{n-1}\into D^n$ gives a
  corresponding decomposition of $B\times QF$ into cells $B\times
  S^{n-1}\to B\times D^n$, which are generating $ij$-cofibrations.  It
  remains to show that~(\ref{eq:product-ij-repl}) is an
  $ij$-equivalence.  For any product projection $B\times C\to B$, we
  have a homeomorphism
  \[\uG(U,B\times C) \iso \Map(U,C),\]
  so applying $\uG(U,-)$ to~(\ref{eq:product-ij-repl}) yields the map
  \begin{equation}
    \Map(U,QF)\to \Map(U,F).\label{eq:sections-of-repl}
  \end{equation}
  Since $QF\to F$ is a $q$-equivalence and $U$ has the homotopy
  type of a CW complex,~(\ref{eq:sections-of-repl}) is also a
  $q$-equivalence.  This is true for all open $U \subset B$, so the
  map~(\ref{eq:product-ij-repl}) is an $ij$-equivalence, as desired,
  and thus an $ij$-cofibrant replacement of $B\times F$.
\end{proof}

The following lemma is stronger than what we need in this section, but
we will use it again in \S\ref{sec:base-change}.

\begin{lemma}\label{lem:ijrepl-hfibt-qeqv}
  Let $f\maps A\to B$ be a map between hereditarily $m$-cofibrant
  spaces, where $B$ is contractible.  Let $X$ be an $h$-fibrant object
  of $\sK/B$ and let $QX\to X$ be an $ij$-cofibrant replacement.  Then
  its pullback $f^*QX\to f^*X$ is both a $q$-equivalence and an
  $ij$-equivalence.
\end{lemma}
In particular, when $f$ is the identity of $B$, this says that $QX\to
X$ itself is also a $q$-equivalence.
\begin{proof}
  Since $f$-equivalences are both $q$-equivalences and
  $ij$-equivalences, and are preserved under pullback, we can work up
  to $f$-equivalence.  For example, since any two $ij$-cofibrant
  replacements for $X$ are $f$-equivalent, it suffices to show the
  result for \emph{some} $ij$-cofibrant replacement.  And since $B$ is
  contractible, any $h$-fibrant $X\to B$ is $f$-equivalent to a
  product projection, so we may as well assume that $X$ itself is a
  product projection $B\times F\to B$.  In this case, we can use as
  our $ij$-cofibrant replacement the map $B\times QF \to B\times F$
  from \autoref{lem:prod-ij-coft}, for some $q$-cofibrant replacement
  $QF\to F$.  But the pullback of this map along any $f\maps A\to B$
  is just $A\times QF\to A\times F$, which is both a $q$-equivalence
  and an $ij$-equivalence (the latter by \autoref{lem:prod-ij-coft}).
\end{proof}

\begin{thm}\label{mstr-into-ijstr}
  If $B$ is a locally compact CW complex, the identity adjunction of
  $\sK/B$ is a right Quillen embedding from the $m$-model structure to
  the $ij$-model structure.
\end{thm}
\begin{proof}
  We have shown in \autoref{lem:mstr-ijstr-quillen} that the identity
  adjunction is Quillen, so it remains to show that if $X\to B$ is
  $h$-fibrant, then $QX\to X$ is a $q$-equivalence (here $Q$ denotes
  $ij$-cofibrant replacement).

  Since $B$ is a CW complex, it is locally contractible, so it has a
  cover $(U_\alpha)$ by contractible open sets with inclusions
  $j_\alpha\maps U_\alpha \into B$.  For any $\alpha$, the functor
  $j_\alpha^*\maps \sK/B\to\sK/U_\alpha$ preserves $ij$-equivalences
  and $ij$-cofibrations, so $j_\alpha^*QX \to j_\alpha^*X$ is again an
  $ij$-cofibrant replacement of a Hurewicz fibration.  But since
  $U_\alpha$ is contractible, \autoref{lem:ijrepl-hfibt-qeqv} tells us
  that $j_\alpha^*QX \to j_\alpha^*X$ is a $q$-equivalence preserved
  under pullbacks.  In particular, if $j\maps U\into U_\alpha$ is any
  open subset, the further restriction $j^*j_\alpha^*QX \to
  j^*j_\alpha^*X$ is also a $q$-equivalence.

  It follows that $QX\to X$ restricts to a $q$-equivalence over all
  open sets in the cover of $B$ consisting of all finite intersections
  of the sets $U_\alpha$.  Since this cover is closed under finite
  intersections by construction, it follows
  from~\cite[1.4]{may:weq-qfib} that $QX\to X$ is also a
  $q$-equivalence, as desired.
\end{proof}

Since the identity is a Quillen equivalence between the $ij$-model
structure and the $m_{ij,f}$-model structure, it follows that the
identity is also a right Quillen embedding from the $m$-model
structure to the $m_{ij,f}$-model structure.

\begin{cor}\label{mstr-into-hosheaves}
  If $B$ is a locally compact CW complex, then the relative
  realization-singular complex adjunction
  \[|-|_B\maps \calS^{\sB\op} \toot \sK/B\spam S^B
  \]
  is a right Quillen embedding from the $m$-model structure to the
  homotopy sheaf model structure.
\end{cor}
\begin{proof}
  All CW complexes are Hausdorff and hereditarily paracompact, so we
  can compose the right Quillen embedding from
  \autoref{mstr-into-ijstr} with the Quillen equivalence from
  \autoref{hosheaf<->ijstr}.
\end{proof}

This result shows that parametrized spaces do, in fact, embed
`homotopically fully and faithfully' into homotopy sheaves.  In
particular, at the level of homotopy categories we have an adjunction
\[\xymatrix{\iota^\star \maps \HoSh(B) \eqv \Ho_{ij}(\sK/B)\;
  \ar@<1mm>[r] &
  \;\Ho_q(\sK/B) \spam \iota_\star
  \ar@<1mm>[l]}\]
in which the right adjoint is full and faithful.  The existence of
$\iota_\star$, though not its full-and-faithfulness, was observed
in~\cite{ij:ex-spaces}.

\begin{rmk}
  Unlike \autoref{hosheaf<->ijstr}, \autoref{mstr-into-hosheaves}
  remains true if we localize $\calS^{\sB\op}$ at all
  \emph{hypercovers}, because the realization of any hypercover
  \emph{is} a $q$-equivalence, though not an $ij$-equivalence---this
  follows from the proof of~\cite[Thm.\ 12.1]{am:etale-homotopy}.  In
  other words, \emph{all locally constant homotopy sheaves are
    hypercomplete}.  Thus,~\cite{toen:locconst} was able to prove a
  simplicial version of \autoref{mstr-into-hosheaves} using a
  localization at all hypercovers.
\end{rmk}

The only property of a locally compact CW complex used in
\autoref{mstr-into-ijstr}, aside from hereditary $m$-cofibrancy, is
that it is locally contractible.  For \autoref{mstr-into-hosheaves},
we also need it to be Hausdorff and hereditarily paracompact.  Thus,
we can abstract the necessary properties of $B$ as follows.

\begin{defn}
  We say that a space is a \textbf{good ancestor} if it is
  \begin{enumerate}
  \item compactly generated,
  \item Hausdorff and hereditarily paracompact,
  \item hereditarily $m$-cofibrant, and
  \item locally contractible.
  \end{enumerate}
\end{defn}

Any locally compact CW complex is a good ancestor.  Moreover, any open
subspace of a good ancestor is a good ancestor; that is, the property
of being a good ancestor is itself hereditary.  This will be important
in \S\ref{sec:base-change}.

Finally, most of the results of this section have corresponding
versions for the sectioned theory, by the following lemma.

\begin{lemma}\label{lem-pointed-qemb}
  If $F\maps \sC\toot \sD\spam G$ is a right Quillen embedding and the
  terminal object of \sC\ is cofibrant and preserved by $F$, then
  $F_*\maps \sC_*\toot \sD_*\spam G_*$ is also a right Quillen
  embedding.
\end{lemma}
\begin{proof}
  Since the terminal object of \sC\ is cofibrant, any cofibrant object
  of $\sC_*$ is also cofibrant in \sC.  The fact that $F$ preserves
  the terminal object implies that the pointed adjunction $F_*\adj
  G_*$ is defined simply by applying $F$ and $G$ to underlying
  objects.  Thus, if $Y$ is fibrant in $\sD_*$, the map $F_* Q_* G_* Y
  \to Y$ is just $FQGY\to Y$, which is a weak equivalence since $Y$ is
  also fibrant in \sD.
\end{proof}

\begin{cor}
  If $B$ is a good ancestor, the identity functor of $\sK_B$ is a
  right Quillen embedding from the $m$-model structure to the
  $ij$-model structure, and the pointed adjunction
  \[|-|_B\maps \calS^{\sB\op}_* \toot \sK_B\spam S^B.
  \]
  is a right Quillen embedding from the $m$-model structure to the
  homotopy sheaf model structure.
\end{cor}
\begin{proof}
  The terminal object is cofibrant in all model structures under
  consideration, the identity functor clearly preserves it, and it is
  easy to see that so does the relative geometric realization.  Thus
  we can apply \autoref{lem-pointed-qemb}.
\end{proof}

Of course, since there is no known `$fp$-model structure', the
statements about the $m_{ij,f}$-model structure have no sectioned
analogue.

\section{The essential image}
\label{sec:essimg}

We would now like to identify the image of the right Quillen embedding
from \autoref{mstr-into-ijstr}.  As explained in
\S\ref{sec:introduction}, our intuition is that it consists of the
locally constant homotopy sheaves.  Of course, we need to make precise
what we mean by `locally constant' in a homotopical sense.  From now
on, we will take the $ij$-structure as our model for homotopy sheaves.

\begin{defn}
  We say that an object $X\to B$ of $\sK/B$ is \textbf{constant}
  if it is isomorphic in $\Ho_{ij}(\sK/B)$ to one of the form
  $B\times F \to B$.  We say that it is \textbf{locally constant} if
  $B$ admits an open cover $(U_\alpha)$, with inclusions
  $j_\alpha\maps U_\alpha\into B$, such that $j_\alpha^*X$ is constant
  for all $\alpha$.
\end{defn}

We have the following trivial observation.

\begin{lemma}
  If $B$ is locally contractible, then any Hurewicz fibration $X\to B$
  is locally constant.
\end{lemma}
\begin{proof}
  Take a cover by contractible opens; then $j_\alpha^*X$ is a
  fibration over a contractible space, hence $f$-equivalent to a
  product projection.
\end{proof}

We observe that the essential image of the embedding $\iota_\star\maps
\Ho_q(\sK/B)\into \Ho_{ij}(\sK/B)$ consists of the objects of
$\Ho_{ij}(\sK/B)$ isomorphic to Hurewicz fibrations, since the latter
are the fibrant objects in the $m$-model structure.  Thus, this image
is contained in the locally constant objects.  We now intend to show
that conversely, any locally constant object of $\Ho_{ij}(\sK/B)$ is
isomorphic to a Hurewicz fibration.  We begin with the following lemma
which clarifies the structure of locally constant objects.

\begin{lemma}\label{lem:ijcoft-qfib}
  Let $B$ be a good ancestor and $X\to B$ be $ij$-cofibrant and
  locally constant.  Then $X$ is locally $f$-equivalent to a product
  projection.  In particular, it is a quasifibration.
\end{lemma}
\begin{proof}
  Since $X$ is locally constant, we have a cover $(U_\alpha)$ such
  that $j_\alpha^*X$ is isomorphic in $\Ho_{ij}(\sK/U_\alpha)$ to a
  product projection $U_\alpha\times F_\alpha$.  Let $QF_\alpha \to
  F_\alpha$ be a $q$-cofibrant replacement; then by
  \autoref{lem:prod-ij-coft}, $U_\alpha\times QF_\alpha$ is an
  $ij$-cofibrant replacement for $U_\alpha\times F_\alpha$.
  Therefore, since $j_\alpha^*X$ is also cofibrant, the composite
  isomorphism $j_\alpha^*X \iso U_\alpha\times QF_\alpha$ in
  $\Ho_{ij}(\sK/B)$ is realized by an $ij$-equivalence in
  $\sK/U_\alpha$.  And since this is an $ij$-equivalence between
  $ij$-cofibrant objects, it is actually an $f$-equivalence.  Thus,
  $X$ is locally $f$-equivalent to a product projection.

  Now, since $j_\alpha^*X$ is $f$-equivalent to an $h$-fibration, it
  is a `halb-fibration'
  (see~\cite{dold:fiberwise-equiv,dold:part1-fib}), and in particular
  a quasifibration.  Since $f$-equivalences and $h$-fibrations are
  preserved by restricting to open subspaces, this is also true of
  $j^* X$ for any open set $j\maps U\into X$ where $U\subset U_\alpha$
  for some $\alpha$, and in particular for finite intersections of the
  $U_\alpha$.  Thus, $B$ has an open cover which is closed under
  finite intersections and over which $X$ is a quasifibration.
  Standard criteria (e.g.~\cite{dold-thom:quasifib}) then imply that
  $X$ itself is a quasifibration.
\end{proof}

\begin{thm}
  If $B$ is a good ancestor, then any locally constant object of
  $\Ho_{ij}(\sK/B)$ is isomorphic in $\Ho_{ij}(\sK/B)$ to a Hurewicz
  fibration.  Therefore, the essential image of $\Ho_q(\sK/B)$ in
  $\Ho_{ij}(\sK/B)$ consists precisely of the locally constant
  objects.
\end{thm}
\begin{proof}
  Let $X$ be locally constant.  Since every object of
  $\Ho_{ij}(\sK/B)$ is isomorphic to an $ij$-cofibrant one, we may
  assume that $X$ is $ij$-cofibrant.  Thus, by
  \autoref{lem:ijcoft-qfib}, there is a cover $(U_\alpha)$ and
  $f$-equivalences
  \begin{equation}
    U_\alpha\times F_\alpha \too j_\alpha^*X.\label{eq:prrepl-over-alpha}
  \end{equation}  
  Moreover, since $B$ is paracompact, we may assume by refinement that
  the cover $(U_\alpha)$ is numerable.

  Now, let $X\to RX$ be an $h$-fibrant replacement; we want to show
  that it is actually an $ij$-equivalence.
  By~\cite[6.1]{ij:ex-spaces}, since the cover $(U_\alpha)$ is
  numerable, it suffices to show that the induced map
  \begin{equation}
    j_\alpha^*X \too j_\alpha^*RX\label{eq:hrepl-over-alpha}
  \end{equation}
  is an $ij$-equivalence for each $\alpha$.  By definition of a
  quasifibration, the map $X\to RX$ induces a $q$-equivalence on all
  fibers, and therefore so does $j_\alpha^*X \to j_\alpha^*RX$.
  Moreover, since~(\ref{eq:prrepl-over-alpha}) is an $f$-equivalence,
  it induces an $h$-equivalence on fibers, and thus the composite
  \begin{equation}
    U_\alpha\times F_\alpha \too j_\alpha^*X \too j_\alpha^*RX\label{eq:comprepl-over-alpha}
  \end{equation}
  induces a $q$-equivalence on all fibers.  But both $U_\alpha\times
  F_\alpha$ and $j_\alpha^*RX$ are $h$-fibrant, so by the five lemma,
  (\ref{eq:comprepl-over-alpha}) is itself a $q$-equivalence.  Again
  since both are $h$-fibrant, \cite[6.5]{ij:ex-spaces} (due to Lewis)
  implies that~(\ref{eq:comprepl-over-alpha}) is an $ij$-equivalence.

  Finally, since~(\ref{eq:prrepl-over-alpha}) is also an
  $ij$-equivalence, the 2-out-of-3 property implies
  that~(\ref{eq:hrepl-over-alpha}) is too.  This shows that $X\to RX$
  is an $ij$-equivalence, and thus $X$ is isomorphic in $\Ho_{ij}(\sK/B)$ to
  the Hurewicz fibration $RX$.
\end{proof}

\section{Base change and homotopy sheaves}
\label{sec:base-change}

We now consider the relationship between the base change functors for
parametrized spaces and for homotopy sheaves.  This is nontrivial
because $f^*$ has a \emph{left} derived functor $\bfL_{sh} f^* \iso
\bfL_{ij} f^*$ for homotopy sheaves but a \emph{right} derived functor
$\bfR_q f^*$ for parametrized spaces.  However, we will prove that the
two agree up to homotopy.  Recall that we write $\iota^\star\maps
\Ho_{ij}(\sK/B) \toot \Ho_q(\sK/B) \spam \iota_\star$ for the right Quillen
embedding from \S\ref{sec:comparison}.

\begin{thm}\label{basechange-iso}
  For any map $f\maps A\to B$ between good ancestors, we have a
  natural isomorphism
  \begin{equation}
    \bfL_{ij} f^* \circ \iota_\star \iso \iota_\star \circ \bfR_q f^*.\label{eq:basechange-iso}
  \end{equation}
  in both the sectioned and unsectioned cases.
\end{thm}
\begin{proof}
  We prove the unsectioned case first.  Since $\iota_\star$ and
  $\bfR_q f^*$ are both right derived functors for the same model
  structures, their composition is just given by their point-set
  composite applied to a fibrant object; in other words, $f^*X$ where
  $X\to B$ is an $h$-fibration.  On the other hand, $\iota_\star X$ is
  again $X$ (when $X$ is $h$-fibrant), but to compute $\bfL_{ij} f^*
  (\iota_\star X)$ we must replace $X$ by an $ij$-cofibrant object
  $QX$.  Since this comes with an $ij$-equivalence $QX \we X$, we have
  a canonical map
  \begin{equation}
    f^* Q X  \to f^* X\label{eq:basechange-comp-we}
  \end{equation}
  which represents a map
  \begin{equation}
    \bfL_{ij} f^* (\iota_\star X) \to \iota_\star (\bfR_q f^* X).\label{eq:basechange-comp-iso}
  \end{equation}
  In the terminology of~\cite{shulman:dblderived}, this is the `derived
  natural transformation' of the point-set level equality $f^*\circ \Id
  = \Id \circ f^*$.

  We claim that~(\ref{eq:basechange-comp-iso}) is an isomorphism, or
  equivalently that~(\ref{eq:basechange-comp-we}) is an
  $ij$-equivalence.  Let $(U_\alpha)$ be a numerable cover of $B$ by
  contractible opens.  Then
  \begin{equation}
    j_\alpha^* QX \to j_\alpha^* X\label{eq:local-repl}
  \end{equation}
  is again an $ij$-cofibrant replacement in $\sK/U_\alpha$.  By
  \autoref{lem:ijrepl-hfibt-qeqv}, since $U_\alpha$ is contractible
  and $X$ is $h$-fibrant, any pullback of~\eqref{eq:local-repl} to
  another good ancestor is an $ij$-equivalence.

  In particular, this applies to the pullback along the restriction
  $f_\alpha\maps f^{-1}(U_\alpha) \to U_\alpha$ of $f$; thus the map
  \[f_\alpha^* j_\alpha^* QX \to f_\alpha^* j_\alpha^* X
  \]
  is an $ij$-equivalence.  But if we write $i_\alpha\maps
  f^{-1}(U_\alpha) \into A$ for the inclusion, then we have $fi_\alpha
  = j_\alpha f_\alpha$ and hence $f_\alpha^* j_\alpha^* \iso
  i_\alpha^* f^*$, so the map
  \[i_\alpha^* f^* QX \to i_\alpha^* f^* X
  \]
  is also an $ij$-equivalence over $f^{-1}(U_\alpha)$ for all
  $\alpha$.  Since the cover $(f^{-1}(U_\alpha))$ of $A$ is also
  numerable, it follows from~\cite[6.1]{ij:ex-spaces} that $f^*QX\to
  f^*X$ is an $ij$-equivalence over $A$, as desired.

  In the sectioned case, we again have a map
  \[\bfL_{ij} f^* (\iota_\star X) \to \iota_\star (\bfR_q f^* X)\]
  represented by the map
  \begin{equation}
    f^* Q X \to f^* X,\label{eq:ptd-basechange-comp-we}
  \end{equation}
  where $X$ is $h$-fibrant in $\sK_B$ and now $Q$ denotes
  $ij$-cofibrant replacement in $\sK_B$.  But if we forget the
  sections, we see that $X$ is also $h$-fibrant in $\sK/B$, and since
  the terminal object of $\sK/B$ is $ij$-cofibrant, $QX$ is also an
  $ij$-cofibrant replacement in $\sK/B$.  Thus, applying the result
  for the unsectioned case, we see
  that~\eqref{eq:ptd-basechange-comp-we} is an $ij$-equivalence in
  $\sK/A$, hence also in $\sK_A$.
\end{proof}

Since $\iota_\star$ has a left adjoint $\iota^\star$, the
isomorphism~(\ref{eq:basechange-iso}) has a `mate'
\begin{equation}\label{eq:basechange-mate-1}
  \iota^\star \circ \bfL_{ij} f^* \too \bfR_q f^*\circ \iota^\star.
\end{equation}
Similarly, since $\bfL_{ij} f^*$ has a right adjoint $\bfR_{ij} f^*$,
and, in the sectioned case, $\bfR_q f^*$ has a partial right adjoint
$\bfM f_*$ defined on connected spaces (obtained using Brown
representability),~(\ref{eq:basechange-iso}) has another `partial
mate'
\begin{equation}\label{eq:basechange-mate-2}
  \iota_\star\circ \bfM f_* \too \bfR_{ij} f_*\circ \iota_\star
\end{equation}
defined on subcategories of connected spaces.  (The `M' may stand
either for `middle' or `mysterious'.)

If $f$ is a bundle of cell complexes, then $f^*$ is also left Quillen
for the $q$-structures, so $\bfR_q f^* = \bfL_q f^*$ also has a
totally defined right adjoint $\bfR_q f_*$.  In this case we have an
analogous transformation
\begin{equation}\label{eq:basechange-mate-3}
  \iota_\star\circ \bfR_q f_* \too \bfR_{ij} f_*\circ \iota_\star
\end{equation}
which is defined everywhere.

Standard categorical arguments show that~(\ref{eq:basechange-mate-2})
or~(\ref{eq:basechange-mate-3}) is an isomorphism if and only
if~(\ref{eq:basechange-mate-1}) is.  Thus, since $\bfM f_*$ is
difficult to get a handle on, it is natural to focus our efforts
on~(\ref{eq:basechange-mate-1}) instead.  The main result is the
following.  This is a special case of the results
of~\cite{shulman:dblderived} regarding mates of derived natural
transformations.

\begin{prp}\label{basechange-mate-ident}
  The transformation~(\ref{eq:basechange-mate-1}) at an $ij$-cofibrant
  space $X$ over $B$ is isomorphic to the map
  \begin{equation}\label{eq:basechange-mate-path}
    f^* X \too f^* RX
  \end{equation}
  where $X\to RX$ is an $h$-fibrant replacement.
\end{prp}
\begin{proof}
  The map~(\ref{eq:basechange-mate-1}) is defined to be the composite
  \begin{equation}
    \iota^\star \circ \bfL_{ij} f^* \too \iota^\star \circ\bfL_{ij} f^*
    \circ\iota_\star \circ \iota^\star \too[\iso] \iota^\star \circ
    \iota_\star \circ\bfR_q f^* \circ \iota^\star \too \bfR_q f^* \circ
    \iota^\star,\label{basechange-mate1-composite}
  \end{equation}
  where the first map is the unit, and the last the counit, of the
  adjunction $\iota^\star\adj \iota_\star$.  We now trace this through
  on the point-set level.

  We start with an $ij$-cofibrant object $X\to B$, so that
  $\iota^\star(\bfL_{ij} f^* X)$ is given simply by $f^* X$.  The
  first map is the unit of $\iota^\star\adj \iota_\star$ at $X$, which
  is just the map $X \to RX$.  Since we must apply $\bfL_{ij} f^*$ and
  then $\iota^\star$ to this, the first map is actually represented on
  the point-set level by
  \begin{equation}
    Qf^* QX \too Qf^* QRX\label{eq:bc-mate1-comp-1}
  \end{equation}
  where $Q$ denotes $ij$-cofibrant replacement.  We have a diagram
  \begin{equation}
    \xymatrix{Qf^* QX \ar[r]^\sim\ar[d] & f^* QX \ar[d] \ar[r]^\sim & f^* X \ar[d]\\
      Q f^* QRX\ar[r]_\sim & f^* QRX \ar[r] & f^* RX,}\label{eq:bc-mate1-comp-2}
  \end{equation}
  where the left-hand vertical arrow is~(\ref{eq:bc-mate1-comp-1}),
  and in which the marked arrows are $ij$-equivalences, the top-right
  one since $X$ is already $ij$-cofibrant.  The diagram commutes by
  the naturality of $Q$ and $R$.

  We must then compose this with the
  isomorphism~(\ref{eq:basechange-iso}) at $RX$, which is obtained by
  applying $f^*$ to the $ij$-cofibrant replacement map $QRX \to RX$.
  In our case, we must then apply $\iota^\star$ to this, which
  involves another $ij$-cofibrant replacement; thus the middle
  isomorphism in~(\ref{basechange-mate1-composite}) is represented on
  the point-set level by
  \begin{equation}
    Q f^* QRX \too Qf^* RX.\label{eq:bc-mate1-comp-3}
  \end{equation}
  We can add this map to~(\ref{eq:bc-mate1-comp-2}), together with
  another square which commutes by naturality, to obtain the
  following.
  \begin{equation}
    \xymatrix{Qf^* QX \ar[r]^\sim\ar[d] & f^* QX \ar[d] \ar[r]^\sim & f^* X \ar[d]\\
      Q f^* QRX\ar[r]_\sim \ar[dr] &
      f^* QRX \ar[r] & f^* RX.\\
      & Q f^* RX \ar[ur]
    }\label{eq:bc-mate1-comp-4}
  \end{equation}

  Finally, we must compose with the counit of $\iota^\star\adj
  \iota_\star$ at $Qf^* RX$, which is simply the map $Qf^*RX \to
  f^*RX$ at the bottom right of~(\ref{eq:bc-mate1-comp-4}).  Hence, by
  the commutativity of~(\ref{eq:bc-mate1-comp-4}), the composite of
  all three is equal to the composite
  \[Q f^* QX \too[\sim] f^* QX \too[\sim] f^* X \too f^* RX.
  \]
  The first two maps are $ij$-equivalences between $ij$-cofibrant
  objects, hence $f$-equiv\-alences and so also $q$-equivalences.
  Thus, modulo these isomorphisms,~(\ref{eq:basechange-mate-1}) is
  equal to~(\ref{eq:basechange-mate-path}).
\end{proof}

This enables us to show easily that~(\ref{eq:basechange-mate-1}) is an
isomorphism in some cases.

\begin{thm}\label{basechange-iso-mate}
  The transformations~(\ref{eq:basechange-mate-1})
  and~(\ref{eq:basechange-mate-2}) are isomorphisms whenever $f$ is a
  $q$-fibration between good ancestors, as
  is~(\ref{eq:basechange-mate-3}) when $f$ is a bundle of cell
  complexes between good ancestors.
\end{thm}
\begin{proof}
  As observed in~\cite[7.3.4]{maysig:pht}, $f^*$ preserves all
  $q$-equivalences when $f$ is a $q$-fibration, and $X\to RX$ is
  certainly a $q$-equivalence.
\end{proof}

However, we can also use \autoref{basechange-mate-ident} to construct
counterexamples in which (\ref{eq:basechange-mate-1}), and
hence~(\ref{eq:basechange-mate-2}), is not an isomorphism.  This
phenomenon is closely related to~\cite[0.0.1]{maysig:pht}.

\begin{ceg}
  Let $f\maps A\to B$ be a map between good ancestors, where $B$ is
  path connected, and let $U\subset B$ be an open set disjoint from
  the image $f(A)$.  Then $U\to B$ is an $ij$-cofibrant object of
  $\sK/B$ and $f^*U = \emptyset$, hence $\iota^\star (\bfL_{ij} f^* U)
  = \emptyset$ as well.  However, since $B$ is path-connected, there
  are paths connecting points in $f(A)$ to points in $U$, so $f^*RU$
  will not be empty; thus $\bfR_q f^*(\iota^\star U)$ is not empty
  and~(\ref{eq:basechange-mate-1}) is not an equivalence.
\end{ceg}

This very general example makes us suspect
that~(\ref{eq:basechange-mate-1}) will not be an isomorphism for
`most' maps $f$.  In fact, any map $f\maps A\to B$ for
which~(\ref{eq:basechange-mate-1}) is an isomorphism must be `almost a
fibration' in the following sense.  Any open $U\subset B$ is
$ij$-cofibrant as an object of $\sK/B$, so
if~(\ref{eq:basechange-mate-1}) is an isomorphism at $U$, the map
\[f^* U \too f^* RU
\]
must be a $q$-equivalence.  But $f^*U$ is just $f^{-1}(U)$, so this
says that the preimage of $U$ is equivalent to its `homotopy
preimage'.  The analogous statement for \emph{points}, rather than
open sets, is what characterizes a quasifibration.  We conjecture
that~(\ref{eq:basechange-mate-1}) being an isomorphism implies that
$f$ is actually a quasifibration, but we have so far been unable to
prove this.

\begin{rmk}
  We noted in \S\ref{sec:model-struct-homot} that when $f$ is an
  embedding, the adjunction $f_!\adj f^*$ is also Quillen for the
  $ij$-model structures.  Therefore, in this case the left derived
  functor $\bfL_{ij} f^*$ is isomorphic to the right derived functor
  $\bfR_{ij} f^*$, so the isomorphism~(\ref{eq:basechange-iso})
  follows formally because all functors involved are Quillen right
  adjoints and they commute on the point-set level.  It follows that
  we also have an isomorphism
  \begin{equation}\label{eq:basechange-mate-lefts}
    \iota^\star \circ \bfL_{ij} f_! \iso \bfL_q f_! \circ \iota^\star.
  \end{equation}
  Moreover, in this case we have a canonical transformation
  \begin{equation}\label{eq:basechange-mate-lefts2}
    \bfL_{ij} f_! \circ \iota_\star \too \iota_\star \circ \bfL_q f_!
  \end{equation}
  which is represented on the point-set level by the composite
  \[f_!(Q_{ij}X) \too f_!X \too R (f_!X).
  \]
  Here $Q_{ij}$ denotes $ij$-cofibrant replacement, $R$ denotes
  $h$-fibrant replacement, and $X$ is assumed $m$-cofibrant and
  $h$-fibrant over $A$ (which is a subspace of $B$).  Since
  $f_!(Q_{ij}X)$ is supported only on $A$, while $R (f_!X)$ is
  supported on all path-components of $B$ which intersect $A$, this
  can only be an $ij$-equivalence if $A$ is a union of path components
  of $B$.
\end{rmk}

We end this section with some remarks about the potential utility of
\autoref{basechange-iso-mate} for computing the mysterious functor
$\bfM f_*$.  The fact that $f^*\adj f_*$ is Quillen for the $q$-model
structures whenever $f$ is a bundle of cell complexes implies that in
this case, $\bfM f_*$ is isomorphic to $\bfR_q f_*$ and thus may be
computed by first applying $q$-fibrant replacement and then the
point-set level functor $f_*$.  In particular, this applies when $f$
is the projection $r\maps A\to *$ for a cell complex $A$, giving a way
to compute `fiberwise generalized cohomology'.

By comparison, \autoref{basechange-iso-mate} tells us that if $f$ is
any $q$-fibration between good ancestors, then $\bfM f_*$ may be
computed by first applying an $h$-fibrant replacement and then the
point-set level $f_*$.  This is slightly better since it applies to
$q$-fibrations which are not necessarily bundles.  However, since our
spaces must essentially be open subspaces of locally compact CW
complexes, it doesn't give a way to compute fiberwise generalized
cohomology for many new base spaces.

\section{$G$-spaces and $BG$-spaces}
\label{sec:g-spaces-bg}

We now consider the homotopy-theoretic version of the equivalence
between locally constant sheaves and $\pi_1$-sets.  Our intuition is
that spaces parametrized over $A$ should be equivalent to spaces with
an action of the `fundamental $\infty$-groupoid' $\Pi_\infty(A)$.
Topologically speaking, at least if $A$ is connected, $\Pi_\infty(A)$
can be represented by the loop space $\Omega A$ (where we choose a
base point arbitrarily).  We can choose a topological model for
$\Omega A$, such as the Moore loop space or the realization of the Kan
loop group, which is a grouplike topological monoid; then $A$ can be
reconstructed, up to $q$-equivalence, as the classifying space of
$\Omega A$.

Moreover, if $A$ is $m$-cofibrant, then so is $\Omega A$
by~\cite{milnor:cw-type}.  Since the homotopy theory of parametrized
spaces is invariant under $q$-equivalences of the base space, it is
harmless to assume that $A$ is $m$-cofibrant.  Thus, for the rest of
this section we make the following assumption.

\begin{assume}\label{assume:g}
  $G$ is a compactly generated $m$-cofibrant grouplike topological
  monoid whose identity is a nondegenerate basepoint (that is, $*\to
  G$ is an $h$-cofibration).
\end{assume}

Of course, we are thinking of $G=\Omega A$ for a connected
$m$-cofibrant space $A$ which admits a nondegenerate basepoint.  We
intend to compare the homotopy theory of spaces with a $G$-action to
the homotopy theory of spaces parametrized over $BG$.  The results in
this section are basically folklore.  A bijection between equivalence
classes can be found in the survey
article~\cite{stasheff:hsp-classif}, and a full equivalence of
homotopy theories using simplicial fibrations can be found
in~\cite{ddk:equivar-she,dk:equivar->fibr}; our use of the $m$-model
structure on $\sK/BG$ will allow us to prove the strong result while
using only topological spaces.

We will also need a model structure on $G\sK$, the category of (left)
$G$-spaces and $G$-equivariant maps.  If $G$ is a topological group
and $\calH$ is a set of closed subgroups of $G$, there is a
cofibrantly generated model structure on $G\sK$ in which the weak
equivalences are the $G$-maps which induce $q$-equivalences on
$H$-fixed point spaces for all $H\in\calH$; we may call this the
$q\calH$-model structure.  This is most frequently used in equivariant
homotopy theory when \calH\ is the set of all closed subgroups of $G$;
see, for example,~\cite{may:alaska}.  However, we will be interested
instead in the case when $\calH$ consists only of the trivial subgroup
$\{e\}$.  We call this the \textbf{$qe$-model structure} and refer to
its weak equivalences as \textbf{$e$-equivalences}.  This model
structure exists for any topological monoid $G$.

We now construct a Quillen equivalence between the $qe$-model
structure on $G\sK$ and the $m$-model structure on $\sK/BG$.  There is
an obvious functor from $G\sK$ to $\sK/BG$ given by the Borel
construction; a $G$-space $X$ is mapped to $EG\times_{G} X =
B(*,G,X)$, equipped with its projection to $BG = B(*,G,*)$.  This
functor has a right adjoint, which takes a space $Y\to BG$ over $BG$
to the space $\Map_{BG}(EG,Y)$ of maps from $EG$ to $Y$ over $BG$,
equipped with the left $G$-action induced from the right action of $G$
on $EG$.  Thus we have an adjoint pair
\begin{equation}\label{eq:g-bg-adjn}
  B(*,G,-)\,\maps\; G\sK \toot \sK/BG \; \spam\, \Map_{BG}(EG,-).
\end{equation}
Since $EG$ is contractible, we can think of $\Map_{BG}(EG,Y)$ as a
`homotopy fiber' of $Y$ which is chosen in a clever way so as to
inherit a strict $G$-action.  Our first observation is that this
intuition is precise when $Y$ is fibrant.

\begin{lemma}\label{eg-htpy-fiber}
  Under \autoref{assume:g}, if $Y\to BG$ is an $h$-fibration, the map
  \begin{equation}\label{eq:eg-htpy-fiber}
    \Map_{BG}(EG,Y) \too \Map_{BG}(*,Y) = \mathrm{fib}(Y),
  \end{equation}
  induced by the inclusion of the basepoint $*\to EG$, is an
  $h$-trivial $h$-fibration.
\end{lemma}
\begin{proof}
  The map $*\to EG$ is an $h$-equivalence, and \autoref{assume:g}
  ensures that it is also an $h$-cofibration.  Thus, since the
  $h$-model structure on \sK\ is monoidal, the induced pullback corner
  map
  \begin{align*}
    \Map(EG,Y) \too& \Map(*,Y) \times_{\Map(*,BG)} \Map(EG,BG)\\
    =&\; Y \times_{BG} \Map(EG,BG)
  \end{align*}
  is an $h$-trivial $h$-fibration.  Since~\eqref{eq:eg-htpy-fiber} is
  the pullback of this map along
  \[i\times q\maps \mathrm{fib}(Y) \too Y \times_{BG} \Map(EG,BG),\]
  where $i\maps \mathrm{fib}(Y)\into Y$ is the inclusion and $q\maps
  *\to \Map(EG,BG)$ picks out the canonical map $EG\to
  BG$,~\eqref{eq:eg-htpy-fiber} is also an $h$-trivial $h$-fibration.
\end{proof}

\begin{thm}\label{g-bg-equiv}
  Under \autoref{assume:g}, the adjunction~\eqref{eq:g-bg-adjn} is a
  Quillen equivalence between the $qe$-model structure and the
  $m$-model structure.
\end{thm}
\begin{proof}
  The $qe$-model structure is cofibrantly generated, so to show
  that~\eqref{eq:g-bg-adjn} is Quillen, it suffices to show that the
  left adjoint takes the generating $qe$-cofibrations and trivial
  cofibrations to $m$-cofibrations and trivial cofibrations.  The
  generating $qe$-cofibrations are the maps
  \[G\times S^{n-1}\to G\times D^n,\]
  which are taken by the Borel construction to
  \[EG\times S^{n-1} \to EG\times D^n.
  \]
  By~\cite[A.6]{may:ei-gp-perm}, $EG$ is $m$-cofibrant because $G$ is.
  Thus, since the $m$-structure is monoidal, these maps are
  $m$-cofibrations.  The case of the generating trivial
  $qe$-cofibrations is analogous.

  We now show that the adjunction is a Quillen equivalence.  Let $X$
  be a $qe$-cofibrant $G$-space and let $Y$ be an $h$-fibrant space
  over $BG$; we must show that a map $f\maps B(*,G,X)\to Y$ is a
  $q$-equivalence if and only if its adjunct $\fhat\maps X\to
  \Map_{BG}(EG,Y)$ is a $q$-equivalence.  Actually, we will show that
  this is true for \emph{any} $G$-space $X$ and any $h$-fibrant $Y$
  over $BG$.

  By~\cite[7.6]{may:csf}, since $G$ is grouplike, $B(*,G,X)\to BG$ is
  a quasfibration.  Therefore, a map $f\maps B(*,G,X)\to Y$ is a
  $q$-equivalence if and only if it induces a $q$-equivalence on
  fibers.  But the fiber of $B(*,G,X)$ (over the base point) is just
  $X$, so this is true if and only if $X\to \mathrm{fib}(Y)$ is a
  $q$-equivalence.  We now have a commutative triangle
  \[\xymatrix{X \ar[rr] \ar[dr]_<>(.5)\fhat &&
    \mathrm{fib}(Y) \\
    &  \Map_{BG}(EG,Y).\ar[ur]}\]
  We have just argued that the horizontal map is a $q$-equivalence
  precisely when $f$ is.  Since the right-hand diagonal map is an
  $h$-equivalence by \autoref{eg-htpy-fiber}, the desired result follows
  from the 2-out-of-3 property.
\end{proof}

It follows that for any connected nondegenerately based $m$-cofibrant
space $A$, we have a chain of equivalences of homotopy categories
\[\Ho_q(\sK/A) \eqv \Ho_q(\sK/B\Omega A) \eqv \Ho_e((\Omega A)\sK).\]
If $A$ is not $m$-cofibrant, we can first replace it by a CW complex
$\Atil$ and obtain a longer chain of equivalences.

\begin{rmk}
  There is also an \textbf{$h$-model structure} on $G\sK$ in which the
  weak equivalences, fibrations, and cofibrations are the equivariant
  homotopy equivalences (where the homotopy inverse and homotopies
  must also be equivariant), equivariant Hurewicz fibrations, and
  equivariant Hurewicz cofibrations.  Any $h$-equivalence is an
  $e$-equivalence and any $h$-fibration is a $qe$-fibration, so there
  is a mixed \textbf{$me$-model structure}, and the
  adjunction~\eqref{eq:g-bg-adjn} can be shown to also be a Quillen
  equivalence between the $me$-model structure and the $m$-model
  structure.
\end{rmk}

Finally, by \autoref{qeqv->pted}, we have a corresponding result in
the sectioned and pointed cases.

\begin{cor}
  Under \autoref{assume:g}, the pointed/sectioned version
  of~\eqref{eq:g-bg-adjn},
  \begin{equation}\label{eq:ptd-g-bg-adjn}
    B(*,G,-)\maps G\sK_* \toot \sK_{BG} \spam \Map_{*,BG}(EG,-),
  \end{equation}
  is also a Quillen equivalence.
\end{cor}

\section{Base change and $G$-spaces}
\label{sec:base-change-g}

We now compare the base change functors for parametrized spaces with
those for $G$-spaces.  If $f\maps G\to H$ is a map of topological
monoids, it induces a restriction functor $f^*\maps H\sK\to G\sK$,
which has both adjoints $f_!$ and $f_*$ given by left and right Kan
extension.  It is easy to see that the adjunction $f_!\adj f^*$ is
Quillen for the $qe$-model structures, since $f^*$ preserves
fibrations and weak equivalences.  We denote the resulting derived
adjunction by $\bfL_e f_! \adj \bfR_e f^*$.

The map $f$ also induces a map $Bf\maps BG\to BH$ and thus the usual
string of adjunctions $(Bf)_! \adj (Bf)^* \adj (Bf)_*$ between
$\sK/BG$ and $\sK/BH$.  As always, the adjunction $(Bf)_!\adj (Bf)^*$
is Quillen for the $q$- and $m$-model structures.  We write
\[\bfL B(*,G,-) \maps \Ho_e(G\sK) \toot \Ho_q(\sK/BG) \spam \bfR \Map_{BG}(EG,-)\]
for the derived equivalence of the Quillen equivalence from
\autoref{g-bg-equiv}.

\begin{thm}
  If $f\maps G\to H$ is a map between topological monoids satisfying
  \autoref{assume:g}, then we have a natural isomorphism
  \begin{equation}\label{eq:g-bg-basechange}
    \bfR \Map_{BG}(EG,-) \circ \bfR_q (Bf)^* \iso
    \bfR_e f^* \circ \bfR \Map_{BH}(EH,-).
  \end{equation}
\end{thm}
\begin{proof}
  By the composability of Quillen adjunctions, we have isomorphisms
  \begin{align*}
    \bfR \Map_{BG}(EG,-) \circ \bfR_q (Bf)^* &\iso
    \bfR \Map_{BG}(EG,(Bf)^*-)\\
    \bfR_e f^* \circ \bfR \Map_{BH}(EH,-) &\iso
    \bfR (f^*\Map_{BH}(EH,-)).
  \end{align*}
  Now, for any space $Y$ over $BH$, there is a canonical morphism
  \begin{equation}
    f^*\Map_{BH}(EH,Y) \too \Map_{BG}(EG,(Bf)^*Y)\label{eq:g-bg-basechange-map}
  \end{equation}
  induced by the map $Ef\maps EG\to EH$ over $Bf$.  Moreover, the
  following triangle commutes.
  \[\xymatrix{f^*\Map_{BH}(EH,Y) \ar[rr] \ar[dr] &&
    \Map_{BG}(EG,(Bf)^*Y) \ar[dl]\\
    & \save[] *{\Map_{BH}(*,Y) = 
      \Map_{BG}(*,(Bf)^*Y) = 
      \mathrm{fib}(Y)} \phantom{\Map_{BG}(*,(Bf)^*Y)} \restore}\]
  By \autoref{eg-htpy-fiber}, the diagonal maps are $q$-equivalences
  when $Y$ is $h$-fibrant, hence in this
  case~\eqref{eq:g-bg-basechange-map} is also a $q$-equivalence.  Thus
  it represents an isomorphism~\eqref{eq:g-bg-basechange} of derived
  functors, as desired.
\end{proof}

\begin{cor}
  We also have a natural isomorphism
  \begin{equation}\label{eq:g-bg-basechange-mate}
    \bfL_q (Bf)_! \circ \bfL B(*,G,-)  \iso
    \bfL B(*,H,-) \circ \bfL_e f_!.
  \end{equation}
\end{cor}

This means that under the identification of $\Ho_e(G\sK)$ with
$\Ho_q(\sK/BG)$, the derived adjunctions of $f_!\adj f^*$ and
$(Bf)_!\adj (Bf)^*$ agree.  It is easy to check that these results
remain true in the pointed/sectioned case.

It follows from general results about diagram categories
in~\cite[\S22]{shulman:htpylim} that the adjunction
\[f^*\maps H\sK\toot G\sK\spam f_*,
\]
while not in general a Quillen adjunction, does have a derived
adjunction.  The right derived functor $\bfR_e f_*$ can be computed
explicitly as a cobar construction:
\[\bfR_e f_* (X) = C(H,G,X).\]
Moreover, since $f^*$ preserves all $e$-equivalences, its left and
right derived functors agree, so we obtain a chain of derived
adjunctions
\[\bfL_e f_! \adj \bfR_e f^* \iso \bfL_e f^* \adj \bfR_e f_*.\]
In particular, $\bfR_e f^*$ has a right adjoint $\bfR_e f_*$.
Since $\bfR_e f^*$ is isomorphic to $\bfR_q (Bf)^*$, it follows
that the latter also has a totally defined right adjoint, without the
need to appeal to Brown representability.  The same is true in the
pointed/sectioned case.

We can use this, in theory, to compute $\bfR_q g_*$ for an arbitrary
map $g\maps A\to D$ between connected base spaces, by passing along
the chain of Quillen equivalences and computing $\bfR_e (\Omega g)_*$.
This procedure may be too complicated to be useful in practice,
however.

\begin{rmk}
  Of course, the restriction to connected base spaces is innocuous in
  the case considered here: since $\sK/(A\sqcup B) \eqv \sK/A \times
  \sK/B$, we can deal with non-connected base spaces by splitting them
  up into their path components.  However, in an equivariant context,
  this restriction becomes more problematic because `connectedness' is
  a subtler notion.  This does not necessarily mean that our intuition
  that spaces over $B$ are equivalent to $\Pi_\infty(B)$-spaces is
  wrong equivariantly, just that our naive approach using loop spaces
  fails.

  The correct equivariant notion of `homotopy sheaf' is likewise
  somewhat unclear.  If $G$ is a topological group and $B$ is a
  $G$-space, an equivariant $ij$-model structure on $G\sK/B$ is
  constructed in~\cite{ij:ex-spaces}.  The weak equivalences (resp.\
  fibrations) are the maps inducing weak equivalences (resp.\
  fibrations) on spaces of $H$-equivariant sections over $U$, whenever
  $H\le G$ is a closed subgroup and $U\subset B$ is an $H$-invariant
  open set.  If we let $\sO_G(B)$ denote the full topological
  subcategory of $G\sK/B$ spanned by the objects $G\times_H U$ for
  such pairs $(H,U)$, then these weak equivalences and fibrations are
  created by the functor
  \begin{equation}\label{eq:equivar-sections}
    \begin{array}{rcl}
      G\sK/B &\too& \sK^{\sO_G(B)\op}\\
      X &\mapsto& \Map_B(-,X),
    \end{array}
  \end{equation}
  where $\sK^{\sO_G(B)\op}$ is the category of topological presheaves
  on the topologically enriched category $\sO_G(B)$.
  Thus,~\eqref{eq:equivar-sections} is right Quillen from the
  $ij$-structure on $G\sK/B$ to the projective model structure on
  $\sK^{\sO_G(B)\op}$.  We may hope to localize the projective model
  structure to make this adjunction into a Quillen equivalence, but
  the correct covers to use are not obvious.

  The theory of parametrized spaces works just as well equivariantly,
  as is evident in~\cite{maysig:pht}, but it is also unclear whether
  it embeds in the theory of equivariant homotopy sheaves sketched
  above.
\end{rmk}

\bibliography{all,shulman}
\bibliographystyle{halpha}

\end{document}